\documentclass[12pt]{article}
\usepackage{fullpage,graphicx,psfrag,amsmath,amsfonts,verbatim}
\usepackage[small,bf]{caption}

\usepackage{stdmath}

\usepackage[colorlinks=true]{hyperref}
\hypersetup{
  citecolor=blue, 
  urlcolor=blue,
  linkcolor=black,
  filecolor=red 
}

\def\={\ = \ }
\def\>{\, > \, }
\def\<{\, < \, }
\def\R{\mathbb{R}}

\def\pvar{x}
\def\dvar{\lambda}
\def\pdvar{\nu}
\def\pspc{X}
\def\dspc{X^*}
\def\pdspc{^*{X}}
\def\sbspc{S}
\def\sbspcp{S^\perp}
\def\psbspc{^\perp{S}}

\def\rtarw{\rightarrow}

\def\linf{\ell_\infty}

\def\lone{\ell_1}
\def\cn{c_0}
\def\Hinf{\mathcal{H}_\infty}

\def\h10{H^1_0}

\def\iso{\simeq}

\title{A Counterexample and Fix\\to a Minimum Distance Duality Theorem}
\author{Michael C. Rotkowitz
\thanks{The author is with the
Institute for Systems Research and the
Department of Electrical and Computer Engineering, 
The University of Maryland, 
College Park MD 20742 USA,
email: \texttt{mcrotk@umd.edu}.}
\thanks{This material is based upon work supported by the National
  Science Foundation under Grant No.~1351674.}
}

\date{}
\renewcommand\footnotemark{}

\begin{document}
\maketitle

\begin{abstract}
We consider dual optimization problems to the fundamental problem of finding the
minimum distance from a point to a subspace.
We provide a counterexample to a theorem which has appeared in the
literature, relating the minimum distance problem to a maximization
problem in the predual space.
The theorem was stated in a series of papers by Zames and Owen in the
early 1990s in conjunction with a non-standard definition, 
which together are 
true, but the theorem is false when assuming standard definitions, as
it would later appear.
Reasons for the failure of this theorem are discussed; in particular,
the fact that the Hahn-Banach Theorem cannot be guaranteed to provide an extension which
is an element of the predual space.
The condition needed to restore the theorem is derived; namely, that the
annihilator of the pre-annihilator return the original subspace of
interest. 
This condition is consistent with the non-standard definition
initially used, and it is further shown to be necessary in a sense.
\end{abstract}

\section{Introduction}

We consider minimum distance problems in real normed vector spaces.
A well-known duality result for these problems is a generalization of
the projection theorem in the dual space.
We consider the validity of similar duality results in the predual space.
This could be of use for problems arising in spaces for which the
predual is much easier to characterize than the dual, including 
functions of total bounded variation $\mathrm{BV}(\Omega)$ \cite{kunisch2004total},
trace-class operators / nuclear operators,
and the Hardy space $\Hinf$ \cite{zames1993duality,djouadi2002optimal}.

After introducing some terminology in Section~\ref{sec:prelims}, we
review the main duality results for minimum distance problems in
Section~\ref{sec:mindist}, and state (as a conjecture) 
a previously stated theorem providing a duality result in the predual space.
We provide counterexamples to this
conjecture in Section~\ref{sec:counter}, discuss reasons why certain
attempted proofs would fail in Section~\ref{sec:discuss}, and then
show how the result can be restored in Section~\ref{sec:fix}.

\section{Preliminaries}
\label{sec:prelims}

We define the main terms and spaces that will be needed.

\begin{defn}
Given a normed space $\pspc$, its \textbf{(normed) dual} is the space of all
bounded linear functionals on $\pspc$, and is denoted by $\dspc$.
\end{defn}

\begin{defn}
Given a normed space $\pspc$, its \textbf{predual} is the normed space 
denoted by $\pdspc$ such that $(\pdspc)^*\iso\pspc$.
\end{defn}
Dual pairings between an element in a normed space and in its dual
space are then denoted by 
$\langle\cdot,\cdot\rangle:\pspc\times\dspc\rtarw\R$.

\begin{defn}
Given a subset of a normed space $\sbspc\subset\pspc$, its \textbf{annihilator},
denoted by \mbox{$\sbspcp\subset\dspc$}, is given by:
\[
\sbspcp~=~\{\dvar\in\dspc~|~\langle\pvar,\dvar\rangle=0~~\forall~\pvar\in\sbspc\},
\] 
and its \textbf{pre-annihilator},
denoted by $\psbspc\subset~\pdspc$, is given by:
\[
\psbspc~=~\{\pdvar\in\pdspc~|~\langle\pdvar,\pvar\rangle=0~~\forall~\pvar\in\sbspc\}.
\] 
\end{defn}

\begin{rem}
This is the standard definition of the pre-annihilator, as given in
\cite{luenberger_1969,megginson2012introduction,bollobas_1999} 
and many other texts. In general, we then have
$\sbspc ~\subset~ ^\perp(\sbspcp) ~\subset~ (\psbspc)^\perp$
\cite{megginson2012introduction}.
The issues addressed in this paper can be viewed as having all stemmed
from the false assumption that these sets are equivalent when $\sbspc$
is a subspace.
\end{rem}

Let $\Z$ represent the set of integers, $\Z_+$ the set of nonnegative
integers, and $\Z_{++}$ the set of positive integers.

For $1\leq p<\infty$, let 
$$\ell_p~=~\bigg\{x:\Z_+\rightarrow\R ~\bigg|~ \sum_{k=0}^\infty ~\abs{x_k}^p <\infty\bigg\}$$
with norm $\norm{x}_p~=~ \big(\sum_{k=0}^\infty ~\abs{x_k}^p\big)^{1/p}$.
Further let
$$\ell_\infty~=~\bigg\{x:\Z_+\rightarrow\R ~\bigg|~ \sup_{k\in\Z_+}~\abs{x_k}<\infty\bigg\}$$
with norm $\norm{x}_\infty~=~\sup_{k\in\Z_+}\abs{x_k}$.
We also define $\cn\subset\linf$ as
\[
\cn~=~\{x\in\ell_{\infty}~|~\lim_{k\rightarrow\infty}x_k=0\}.
\]
The pairings between any sequences that we will consider will be given
by:
\[
\langle\pvar,\dvar\rangle~=~\sum_{k=0}^\infty~\dvar_k~\pvar_k.
\]

\section{Minimum Distance Results}
\label{sec:mindist}

We review the minimum distance problem and its main duality results.

We first state a theorem which relates finding the minimum distance from a
point to a subspace with a maximization problem in the dual
space. This can be thought of as a generalization of the projection
theorem in Hilbert spaces, and is the first of the two main theorems
in \cite[Ch.~5]{luenberger_1969}. The use of $\max$ indicates that
there is a variable for which the $\sup$ is realized.
\begin{thm}
\label{thm:luen1}
Let $\pspc$ be a real normed linear space, let
$\sbspc\subset\pspc$ be a subspace,
and let $y\in\pspc$. Then:
 \[
 \inf_{\pvar\in\sbspc}~\norm{y-\pvar}~=~
  \max_{\substack{\dvar\in\sbspcp\\[2pt] \norm{\dvar}\leq 1}}~\langle y,\dvar\rangle.
 \]
\end{thm}
We now state the other main theorem of \cite[Ch.~5]{luenberger_1969},
which instead relates a minimum distance problem in the dual space
with a maximization problem in the original/primal space. The use of $\min$
similarly indicates that there is a variable for which the $\inf$ is realized.
\begin{thm}
\label{thm:luen2}
Let $\pspc$ be a real normed linear space, let
$\sbspc\subset\pspc$ be a subspace,
and let $\zeta\in\dspc$. Then:
 \[
 \min_{\dvar\in\sbspcp}~\norm{\zeta-\dvar}~=~
  \sup_{\substack{\pvar\in\sbspc\\[2pt] \norm{\pvar}\leq 1}}~\langle\pvar,\zeta\rangle.
 \]
\end{thm}

This theorem was restated 
in \cite{owen1992robust,owen1993duality,zames1993duality}, 
in a manner which was then repeated 
in several papers
including 
\cite{djouadi2002optimal,djouadil2001multiobjective,djouadi2003optimal,djouadi2004operator,djouadi2004mimo,djouadi2007computation,djouadi2008duality,djouadi_tv_jco,djouadi2014duality}.
The basic idea was that given this relation between a problem in the
dual space and one in the primal space, it could be rewritten as a
relation between a problem in the primal space and one in the predual
space. We now take the theorem from those papers and state it as a conjecture.
\begin{conj}
\label{conj}
Let $\pspc$ be a real normed linear space, let
$\sbspc\subset\pspc$ be a subspace,
and let $y\in\pspc$. Then:
 \[
 \min_{\pvar\in\sbspc}~\norm{y-\pvar}~=~
  \sup_{\substack{\pdvar\in\psbspc\\[2pt]\norm{\pdvar}\leq 1}}~\langle\pdvar,y\rangle.
 \]
\end{conj}

Looking at Theorem~\ref{thm:luen1} along with Conjecture~\ref{conj},
we see that given a problem of finding the minimum distance from a
point to a subspace, we could choose to solve an equivalent
maximization problem in either the dual or the predual space, over
either the annihilator or the pre-annihilator, with the only
difference being that the maximum is only guaranteed to be realized in
the dual space. 

Note that in some of these papers, including 
\cite{owen1992robust,owen1993duality,zames1993duality},
the conjecture was stated along with a non-standard definition of the
pre-annihilator.
This indeed alters its veracity, and is discussed further in
Remark~\ref{other_defn}.

\section{Counterexample}
\label{sec:counter}

We now present a counterexample to the conjecture.

\begin{ex}
\label{ex1}
Let $\pspc=\lone$. It is then
well-established~\cite{luenberger_1969,megginson2012introduction}
that $\dspc\iso\linf$ and that $\pdspc\iso\cn$.

Now consider the subspace $S\subset\pspc$, 
motivated by an exercise in \cite{megginson2012introduction},
given by:
\[
S\=\bigg\{x\in\lone~\bigg|~\sum_{k=0}^\infty~x_k=0\bigg\}.
\]
\begin{align*}
\intertext{Then, considering $x\in S$ with $x_0=-1,~x_i=1$ for some
  $i\in\Z_{++},$ and~ $x_j=0~~\forall~j\notin\{0,i\}$,}
\dvar\in\sbspcp ~&\implies~ \langle\pvar,\dvar\rangle~=~\sum_{k=0}^\infty~\dvar_k~x_k ~=~ \dvar_i-\dvar_0 ~=~ 0\\
~&\implies~ \dvar_i=\dvar_0.
\intertext{Since $i\in\Z_{++}$ was arbitrary, we have $\dvar$ constant
  as a necessary condition for $\dvar\in\sbspcp$. Then, taking any
  $x\in S$, and any $\alpha\in\R$:}
\dvar_k=\alpha~~\forall~k\in\Z_+ ~&\implies~
                                   \langle\pvar,\dvar\rangle~=~\sum_{k=0}^\infty~\dvar_k~x_k
                                   ~=~ \alpha~\sum_{k=0}^\infty~x_k ~=~ 0\\
~&\implies~ \dvar\in\sbspcp.
\end{align*}
So this is a sufficient condition as well, and we have the annihilator
as:
\[
\sbspcp\=\{\dvar\in\linf~|~\dvar_k=\alpha~~\forall~k\in\Z_+ 
\text{ for some }\alpha\in\R\}.
\]
In trying to work out the pre-annihilator, we would follow the same
steps (just with the order of the pairing reversed), and draw the same
conclusion, but we further need $\psbspc\subset\cn$. There is only one
element which is constant, and for which the sequence converges to
zero, and so we have:
\[
\psbspc=\{0\}.
\]

Now, consider $y\in\lone$ such that $y_0=1$ and $y_k=0$ for all
$k\in\Z_{++}$; i.e., $y=\{1,0,0,0,\ldots\}$ 
(actually, any ~$0\neq y\in\lone$ will do).

If $\pvar\in S$, then,
\[
\sum_{k=0}^\infty x_k = 0
~\implies~
\bigg|\sum_{k=1}^\infty x_k\bigg| = \abs{x_0}
~\implies~
\sum_{k=1}^\infty \abs{x_k} \geq \abs{x_0},
\]
and then:
\begin{align*}
\norm{y-x}_{\lone} ~&=~ \abs{1-x_0}+\sum_{k=1}^\infty\abs{x_k}\\
~&\geq~ \abs{1-x_0} + \abs{x_0}\\
~&\geq~ 1.
\end{align*}
This bound is clearly achievable by choosing $x=0$, and so we have:
\[
\min_{\pvar\in\sbspc}~\norm{y-\pvar}~=~1.
\]
But,
\[
\sup_{\substack{\pdvar\in\psbspc\\[2pt]\norm{\pdvar}\leq 1}}\langle\pdvar,y\rangle~=~0,
\]
and so Conjecture~\ref{conj} is false.
\end{ex}

This may raise the question of whether a counterexample can be found
for which the pre-annihilator is not trivial, especially since the
supremum would typically be approached as the norm of the predual
variable approached one.
We thus also provide the following adjusted example.

\begin{ex}
\label{ex2}
Again let $\pspc=\lone$, with $\dspc\iso\linf$ and $\pdspc\iso\cn$.

Now consider the subspace $S\subset\pspc$
given by:
\[
S\=\bigg\{x\in\lone~\bigg|~x_0=0, ~\sum_{k=1}^\infty~x_k=0\bigg\}.
\]
\begin{align*}
\intertext{By similar arguments, we find that the annihilator is:}
\sbspcp~&=~\{\dvar\in\linf~|~\dvar_k=\alpha~~\forall~k\in\Z_{++} 
\text{ for some }\alpha\in\R\};\\
\intertext{that is, the first element may be any real number, and it
  must be some constant thereafter, and that the pre-annihilator is:}
\psbspc~&=~\{\dvar\in\linf~|~\dvar_k=0~~\forall~k\in\Z_{++} \};
\intertext{that is, the first element may be any real number, and it
          must be zero thereafter.}
\end{align*}
We then consider $y\in\lone$ such that $y_0=1, ~y_1=1$ and $y_k=0$ for all
$2\leq k\in\Z_{+}$; i.e., $y=\{1,1,0,0,0,\ldots\}$. 
Then, if $x\in\sbspc$:
\begin{align*}
\norm{y-x}_{\lone} ~&=~\abs{1-x_0}+\abs{1-x_1}+\sum_{k=2}^\infty\abs{x_k}\\
~&=~1+\abs{1-x_1}+\sum_{k=2}^\infty\abs{x_k}\\
~&\geq~1+\abs{1-x_1} + \abs{x_1}\\
~&\geq~ 2,
\end{align*}
and so, since this bound can also be achieved with $x=0$,
\[
\min_{\pvar\in\sbspc}~\norm{y-\pvar}~=~2.
\]
But,
\[
\sup_{\substack{\pdvar\in\psbspc\\[2pt]\norm{\pdvar}\leq 1}}\langle\pdvar,y\rangle
~=~\sup_{\norm{\pdvar}\leq 1}\pdvar_0
~=~1.
\]
\end{ex} 

\subsection{Discussion}
\label{sec:discuss}

We briefly discuss where one would fail if attempting to prove the
conjecture
by following steps similar to those in the proofs of
Theorem~\ref{thm:luen1} or Theorem~\ref{thm:luen2}.

In either case, the same steps can be used to establish that the
minimization will result in a value greater than or equal to that
of the maximization, and the problem arises when trying to
establish that a particular variable exists for one which will achieve
the value of the other, which would then establish the equality.

In the proof of Theorem~\ref{thm:luen1}, a linear functional of unit
norm which vanishes on $S$ is defined
on the subspace $[y+S]=\{\alpha y+x~|~\alpha\in\R,~x\in S\}$,
which is then extended to the entire space $\pspc$ via the Hahn-Banach
theorem. This provides a $\dvar\in\sbspcp$ of unit norm such that 
$\langle y,\dvar\rangle=\inf_{x\in\sbspc}\norm{y-x}$, establishing the
equality. The unrecoverable problem when trying to apply the framework
of this proof to the conjecture, is that the Hahn-Banach theorem
provides an extension to a functional on the primal space, and thus to
an element of the dual space, but cannot be guaranteed to provide an
extension to an element of the predual space.

Suppose we take the value of the supremum to be $m$, and instead try
to adapt the framework of the proof of Theorem~\ref{thm:luen2} to
prove the conjecture.
We can define $\tilde{y}:\psbspc\rtarw\R$ by restriction $\tilde{y}=y|_{\psbspc}$,
such that $\norm{\tilde{y}}=m$.
We then can indeed use the Hahn-Banach theorem 
to extend this to $\check{y}:\pdspc\rtarw\R$, 
i.e., to $\check{y}\in\pspc$,
with $\norm{\check{y}}=m$,
and with $\langle\pdvar,\check{y}\rangle=\langle\pdvar,y\rangle$ 
  if $\pdvar\in\psbspc$.
If we then let $\hat{x}=y-\check{y}$, we get
$\norm{y-\hat{x}}=\norm{\check{y}}=m$ as desired.
The problem then arises when trying to verify that $\hat{x}\in\sbspc$.
Given an arbitrary $\pdvar\in\psbspc$,  this implies that
$\langle\pdvar,\hat{x}\rangle  = \langle\pdvar,y-\check{y}\rangle =0$,
but this only establishes that
$\hat{x}\in(\psbspc)^\perp$, which in general is a superset of $\sbspc$.

\section{Fix}
\label{sec:fix}

We now state a few versions of the fixed duality relation.

\begin{thm}
\label{thm:full}
Let ~$\pspc$ be a real normed linear space, let
~$\sbspc\subset\pspc$ be a subspace,
and let ~$y\in\pspc$. Then:
 \[
 \inf_{\pvar\in\sbspc}~\norm{y-\pvar}~\geq~
 \min_{\pvar\in(\psbspc)^\perp}~\norm{y-\pvar}~=~
  \sup_{\substack{\pdvar\in\psbspc\\[2pt]\norm{\pdvar}\leq 1}}\langle\pdvar,y\rangle.
 \]
\end{thm}
\begin{proof}
The inequality follows immediately from the fact that 
$\sbspc\subset(\psbspc)^\perp$. 
The equality follows immediately from Theorem~\ref{thm:luen2}.
\end{proof}
This has the following obvious and more useful corollary.
\begin{cor}
\label{cor}
Let ~$\pspc$ be a real normed linear space, let
~$\sbspc\subset\pspc$ be a subspace
such that $(\psbspc)^\perp=\sbspc$,
and let ~$y\in\pspc$. Then:
 \[
 \min_{\pvar\in\sbspc}~\norm{y-\pvar}~=~
  \sup_{\substack{\pdvar\in\psbspc\\[2pt]\norm{\pdvar}\leq 1}}\langle\pdvar,y\rangle.
 \]
\end{cor}

\begin{rem}
This condition is later verified to hold for the subspace of interest
in many of the papers which used the conjecture, 
including \cite{owen1992robust,owen1993duality,zames1993duality}, 
so the main results therein are unaffected.
The condition should obviously be checked for other papers which
used the conjecture as well.
\end{rem}

\begin{rem}
\label{other_defn}
Occasionally, the pre-annihilator has been defined as the set
$\hspace{1pt}\psbspc$ for which $(\psbspc)^\perp\iso\sbspc$.
This includes the initial series of papers with the conjecture
\cite{owen1992robust,owen1993duality,zames1993duality}, 
as well as
\cite{djouadil2001multiobjective,djouadi2003optimal,djouadi2004mimo,djouadi2008duality,djouadi_tv_jco}.
While this does avoid the problem with the conjecture, 
and ultimately makes it equivalent to Corollary~\ref{cor},
it is not the standard definition, 
and as we can see from the counterexamples, 
it is not always commensurate with the standard definition.
In Example~\ref{ex1}, for instance, $(\psbspc)^\perp\iso\pspc$.
\end{rem}

\begin{rem}
Note that closure is not enough to guarantee that this condition will
hold, as is already evident from the counterexamples, in which $\sbspc$
was indeed a closed subspace. It follows from $\sbspc$ being a closed
subspace that $\hspace{1pt}\sbspc=\hspace{1pt}^\perp(\sbspcp)$ 
\cite[p.~118]{luenberger_1969}, but in
general we still have $\hspace{1pt}\sbspc\subset(\psbspc)^\perp$.
\end{rem}

Assuming closure does allow our condition to become necessary as well as
sufficient for this strong duality to hold for any point, as stated in
our final theorem.
\begin{thm}
Let ~$\pspc$ be a real normed linear space, and let
~$\sbspc\subset\pspc$ be a closed subspace.
Then:
 \[
 \min_{\pvar\in\sbspc}~\norm{y-\pvar}~=~
  \sup_{\substack{\pdvar\in\psbspc\\[2pt]\norm{\pdvar}\leq 1}}\langle\pdvar,y\rangle
 \]
for all ~$y\in\pspc$, if and only if ~$(\psbspc)^\perp=\sbspc$.
\end{thm}
\begin{proof}
Sufficiency follows from Corollary~\ref{cor}.

Now assume that $\sbspc\subset (\psbspc)^\perp$ is a proper subset, and
choose $y\in (\psbspc)^\perp \setminus \sbspc$.
It follows from Theorem~\ref{thm:full} that
\[
  \sup_{\substack{\pdvar\in\psbspc\\[2pt]\norm{\pdvar}\leq 1}}\langle\pdvar,y\rangle
\=  \min_{\pvar\in(\psbspc)^\perp}~\norm{y-\pvar}
\=0.
\]
If we had $\inf_{x\in\sbspc}\norm{y-x}=0$, then for any $\eps>0$,
there would exist $x\in\sbspc$ such that \mbox{$\norm{y-x}\leq\eps$},
making $y$ a limit point of $\sbspc$, 
and since $\sbspc$ is closed 
this gives $y\in\sbspc$,
a contradiction.
\end{proof}

\bibliographystyle{abbrv}
\bibliography{/Users/mcrotk/Documents/latex_files/rock}
\end{document}